\documentclass[11pt]{amsart}

\usepackage{amsmath, amssymb, amsthm, mathrsfs}
\usepackage{hyperref}
\usepackage{geometry}
\usepackage{mathtools}
\usepackage{xcolor}

\newtheorem{theorem}{Theorem}[section]
\newtheorem{proposition}[theorem]{Proposition}
\newtheorem{lemma}[theorem]{Lemma}
\newtheorem{conj}[theorem]{Conjecture}
\newtheorem{corollary}[theorem]{Corollary}
\theoremstyle{definition}
\newtheorem{definition}[theorem]{Definition}
\theoremstyle{remark}
\newtheorem{remark}[theorem]{Remark}
 \newtheorem*{ack*}{Acknowledgements}
 \newtheorem*{fund*}{Funding}
 \newtheorem*{coi*}{Conflict of interest}
\newtheorem*{ai*}{AI-use Disclosure}

\DeclareMathOperator{\sing}{sing}

\DeclareMathOperator{\Int}{Int}
\DeclareMathOperator{\Lip}{Lip}
\DeclareMathOperator{\osc}{osc}

\DeclareMathOperator{\Sym}{Sym}
\DeclareMathOperator{\tr}{tr}

\DeclareMathOperator{\Span}{span}

\newcommand{\R}{\mathbb{R}}

\newcommand{\dist}{\mathrm{dist}}

\usepackage[
  backend=bibtex,
  sorting=nyt,
  maxnames=50,
  doi=false,isbn=false,url=false,eprint=false,
  giveninits=true
]{biblatex}
\title[Regularity for viscosity solutions to the MSS]{Regularity theorems for Lipschitz viscosity solutions to the minimal surface system}
\author{Bryan Dimler}
\date{\today}
\address{Department of Mathematics \\ University of Notre Dame \\ Notre Dame, IN 46556 \\ USA}
\email{bdimler@nd.edu}

\begin{document}
\begin{abstract}
We prove several regularity theorems concerning Lipschitz viscosity solutions to the minimal surface system. The main lemma (i.e. ``Decoupling Lemma") says that the component functions of such solutions satisfy the homogeneous Pucci inequalities. This allows us to apply the regularity theory for scalar viscosity solutions for elliptic PDE component-wise to a viscosity solution for the system. 

As a result of the Decoupling Lemma, we deduce that 2D Lipschitz viscosity solutions to the system are analytic in the interior, and obtain a quantitative partial regularity theorem. We also prove a geometric Harnack inequality for Lipschitz solutions. The Harnack inequality yields a strong maximum principle and half-space Liouville theorem. Another significant consequence of the lemma is a boundary regularity theorem for Lipschitz viscosity solutions. Notably, our theorem does not require any density assumption at the graph boundary, and we do not assume stationarity. This is used to prove a classification theorem for the singular set of a 4D stationary solution on a bounded domain. We also prove an $\varepsilon$-regularity theorem extending a result of the author. 

To conclude, we survey open problems. As an example, we have included a Bernstein-type theorem for 3D homogeneous entire Lipschitz viscosity solutions. Taken together, our results suggest that viscosity methods are a useful tool for studying the regularity of minimal graphs.
\end{abstract}

\maketitle
`
\section{Introduction}
Suppose $\Omega \subset \mathbb{R}^n$ ($n \geq 2$) is a connected domain with Lipschitz boundary. Let $u \in \Lip(\Omega; \mathbb{R}^m) \cap C^0(\overline{\Omega};\mathbb{R}^m)$ for some $m \geq 2$. We will denote the graph of $u$ in $\mathbb{R}^{n+m}$ by 
\[\mathscr{G}_u := \{(x,u(x)) \in \mathbb{R}^{n} \times \mathbb{R}^m : x \in \overline{\Omega} \subset \mathbb{R}^n\} \subset \mathbb{R}^{n+m}.\]
The submanifold $\mathscr{G}_u$ is an $n$-dimensional embedded Lipschitz submanifold in $\mathbb{R}^{n+m}$, whose $n$-dimensional surface area is given by the integral
\[\mathcal{A}(\mathscr{G}_u) := \int_{\mathscr{G}_u} \, d\mathcal{H}^n = \int_\Omega \sqrt{\det\big(I + Du^T(x) Du(x)\big)} \, dx.\]
Here, $\mathcal{H}^n$ is the $n$-dimensional Hausdorff measure on $\mathbb{R}^{n+m}$ restricted to $\mathscr{G}_u$. 

If $\mathscr{G}_u$ is \emph{stationary} (i.e. a critical point of $\mathcal{A}$) and $\psi_u := (x,u(x))$ is the \emph{graph map}, then $\psi_u$ is an integral weak solution to the \emph{minimal surface system} (MSS):
\begin{equation}\label{MSS1}
    \Delta_g \psi_u(x) = 0 \text{ in } \Omega,
\end{equation}
where $\Delta_g$ is the Laplace-Beltrami operator on $\mathscr{G}_u$, defined component-wise, and $g := I + Du^T Du$ is the canonical metric for $\mathscr{G}_u$. In this case, we call $u$ a \emph{stationary solution}. When $u \in C^2(\Omega; \mathbb{R}^m) \cap C^0(\overline{\Omega}; \mathbb{R}^m)$, the equation \eqref{MSS1} reduces to 
\begin{equation}
    g^{ij}(x)u_{x_i x_j}^\beta(x) = 0 \text{ in } \Omega \text{ for each } \beta = 1,\ldots,m. \label{MSS2}
\end{equation}
The functions $g^{ij}$ represent the components of $g^{-1}$ and the $u^\beta$ are the component functions for $u$. Any $u$ solving \eqref{MSS2} is called a \emph{classical solution}. As \eqref{MSS1} encodes the vanishing of the mean curvature vector for $\mathscr{G}_u$, graphs of classical solutions are minimal submanifolds. When $m=1$, \eqref{MSS1} is simply the scalar \emph{minimal surface equation} (MSE). 

In this paper, we study the regularity of Lipschitz viscosity solutions to the MSS as defined by Savin \cite{Sa18}. Let us motivate the definition, which is based on the \emph{barrier principle} for minimal submanifolds.\footnote{The barrier principle says that an $n$-dimensional minimal submanifold cannot be touched from one side at an interior point by a hypersurface that is strictly mean $n$-convex there (e.g. \cite[Theorem 1.2]{Wh16}).} If $u \in C^2(\Omega; \mathbb{R}^m) \cap C^0(\overline{\Omega}; \mathbb{R}^m)$ is a classical solution, $H \in C^2(\mathbb{R}^{n+m})$, and $\mathscr{G}_u$ touches $\{H = c\}$ from the side $\{H \leq c\}$ at $p_0 := (x_0, u(x_0)) \in \Int \mathscr{G}_u$, then $H\lvert_{\mathscr{G}_u} = H \circ \psi_u$ achieves its maximum at $x_0 \in \Omega$. Thus, we have
\[\Delta_g H(p_0) \leq 0 \text{ with } \nabla H(p_0) \perp T_{p_0} \mathscr{G}_u.\]
Since $\mathscr{G}_u$ is minimal, $\Delta_g H$ coincides with the Laplacian of $H$ when $H$ is restricted to $T_{p_0} \mathscr{G}_u$. 

Due to the preceding discussion, it makes sense to define a \emph{viscosity solution for the MSS} to be any $u \in C^0(\overline{\Omega};\mathbb{R}^m)$ satisfying the following property: if $H \in C^2(\mathbb{R}^{n+m})$ and $\mathscr{G}_u$ touches $\{H = c\}$ from the side $\{H \leq c\}$ at $p_0 \in \Int \mathscr{G}_u$, then there is an $n$-plane $T \subset \mathbb{R}^{n+m}$ such that $T \perp \nabla H(p_0)$ and
\[\Delta_TH(p_0) \leq 0,\]
where $\Delta_T H$ denotes the Laplacian of $H \lvert_T$. That is, $u$ is a viscosity solution to the MSS if $\mathscr{G}_u$ cannot be touched by a $C^2$ hypersurface from one side which is strictly mean $n$-convex. We call such a touching surface a \emph{comparison surface}, and its defining function $H$ a \emph{comparison function}. The definitions are consistent with the barrier principle. When $u$ is a viscosity solution, we say it solves the MSS in the \emph{viscosity sense}. For the sake of simplicity, we will only be concerned with Lipschitz viscosity solutions to the MSS.

The main contribution of this paper is a decoupling principle for Lipschitz viscosity solutions to the MSS (i.e. the ``Decoupling Lemma"). Although the MSS is a genuinely coupled nonlinear system, we show that each component of a Lipschitz viscosity solution satisfies a scalar uniformly elliptic equation, with ellipticity constants depending only on the Lipschitz constant of the solution. This allows us to apply the regularity theory for scalar elliptic PDE component-wise. Precisely:
\begin{lemma}[Decoupling Lemma]\label{lem: decoupling}
    Suppose $u := (u^1,\ldots, u^m) \in \Lip(\Omega;\mathbb{R}^m) \cap C^0(\overline \Omega; \mathbb{R}^m)$ solves the MSS in $\Omega$ in the viscosity sense. Then each component function $u^\beta$ of $u$ satisfies the homogeneous Pucci inequalities (see \eqref{eq: homoPucci}) in $\Omega$ with ellipticity constants depending only on $\Lip u$.
\end{lemma}
The proof of the Decoupling Lemma uses White’s theory of \((n,h)\)-sets \cite{Wh16} and, in particular, a Lusin-type theorem for the generalized normal bundle of such sets due to Santilli (see Theorem \ref{thm: santilli}). We first observe that every viscosity graph is an \((n,0)\)-set. In fact, the definitions are equivalent (Proposition \ref{prop: euivalent}). Santilli’s Lusin theorem then allows us to tilt a touching hypersurface so that, for almost every small affine tilt, contact occurs at a regular point of the solution's graph. Hence, their theorem may be viewed as a geometric analogue of the measure-theoretic mechanism underlying ABP arguments, where the contact-point-to-vertex map is Lipschitz and therefore sends null sets to null sets. In this sense, the proof of Lemma \ref{lem: decoupling} can be seen as a geometric ABP argument. We refer the reader to \cite[Chapter 3]{CC95} for more on this perspective.

Combined with the scalar regularity theory and the observations of \cite[Section 4]{BD2}, Lemma \ref{lem: decoupling} yields several new regularity and rigidity results for Lipschitz viscosity solutions. We briefly state the most prominent ones. The first, which is immediate from a theorem of Jensen for 2D scalar equations \cite{Je96}, is full interior regularity for 2D solutions for the system.
\begin{theorem}[2D Regularity] \label{thm: 2d}
  Suppose $\Omega \subset \mathbb{R}^2$. If $u \in \Lip(\Omega; \mathbb{R}^m)\cap C^0(\overline{\Omega};\mathbb{R}^m)$ solves the MSS in $\Omega$ in the viscosity sense, then $u$ is analytic in $\Omega$.
\end{theorem}

By modifying Armstrong, Silvestre, and Smart's partial regularity argument for fully nonlinear scalar equations \cite[Theorem 1]{ASS12}, we obtain a quantitative partial regularity theorem for the system. This sharpens the qualitative version proved in \cite[Theorem 4.11]{BD2}. 
\begin{theorem}[Quantitative Partial Regularity] \label{thm: partialreg}
    Assume $\Omega \subset \mathbb{R}^n$ for $n \geq 3$. Suppose $u \in \Lip(\Omega;\mathbb{R}^m) \cap C^0(\overline{\Omega}; \mathbb{R}^m)$ solves the MSS in $\Omega$ in the viscosity sense. Then $\sing(u)$ has Hausdorff dimension at most $n-\varepsilon$, where 
    \[\varepsilon
\geq  \frac{1}{4\log n}
\left(
\frac{3+5M^2}{3(1+M^2)^2}
\right)^{n-1}  \text{ for } \Lip u \leq M.
\]
\end{theorem}
\begin{remark}
    The lower bound on $\varepsilon$ is due to \cite[Theorem 1]{NT23} (see also \cite{Mo19}). A sharp lower bound for $\varepsilon$ is of interest since, if $\varepsilon > 1$, Lipschitz viscosity solutions are necessarily stationary by removal of singularities \cite[Corollary 4.9]{BD2}. However, since the bound on $\varepsilon$ is precisely the bound in the $W^{2,\varepsilon}$ estimate for fully nonlinear scalar equations, our current proof forces $\varepsilon < 1$. Thus, a new approach is needed if one hopes to show $\varepsilon > 1$ for a general Lipschitz viscosity solution.
\end{remark}

The reader may notice that Theorem \ref{thm: partialreg} is weaker than the corresponding partial regularity theorem for stationary solutions (e.g. \cite[Theorem 3.7]{BD2}).\footnote{The $n-4$ dimension bound is originally due to Duggan \cite[Theorem 4.4]{Du86}, though their proof uses a GMT argument while \cite[Theorem 3.7]{BD2} is a PDE argument. The methods in this paper are based on those in \cite{BD2}.} This is because, in our setting, one does not have access to a monotonicity formula and must rely on the general theory for scalar equations.

A simple rotation of coordinates argument yields a Harnack inequality, which implies a geometric strong maximum principle (Corollary \ref{cor: max}) and a half-space Liouville theorem (Corollary \ref{cor: liouville}). 
\begin{proposition}[Harnack Inequality]\label{prop: harnack}
    Suppose $u\in \Lip(\Omega; \mathbb{R}^m) \cap C^0(\overline \Omega; \mathbb{R}^m)$ solves the MSS in $\Omega$ in the viscosity sense. For fixed $(a,b) \in \mathbb{R}^n \times \mathbb{R}^m$ and $c \in \mathbb{R}$, consider the affine function $L : \mathbb{R}^{n} \times \mathbb{R}^m \rightarrow \mathbb{R}$
    \begin{equation*}
        L(x,z) := a \cdot x + b \cdot z + c
    \end{equation*}
    and set $w(x) := L(x,u(x))$. If $w \geq 0$ in $B_{2r}^n \subset \Omega$, then there is a constant $C:= C(n,\Lip u)$ such that \[\sup_{B_r^n} w \leq C \inf_{B_r^n} w.\]
\end{proposition}

A particularly significant consequence we will prove is the following boundary regularity theorem. For $\delta > 0$, we set 
\[\Omega_\delta := \{x \in \Omega : \dist(x, \partial \Omega) < \delta\}.\] 
\begin{theorem}[Boundary Regularity]\label{thm: boundary}
    Assume $\Omega \subset \mathbb{R}^n$ is a bounded domain with $C^2$ boundary. Suppose $u \in \Lip(\Omega; \mathbb{R}^m) \cap C^0(\overline{\Omega};\mathbb{R}^m)$ solves the MSS in $\Omega$ in the viscosity sense, and that $u \lvert_{\partial \Omega} = \phi \in C^{1,\alpha}(\partial \Omega; \mathbb{R}^m)$ for some $\alpha \in (0,1)$. Then there are constants $\delta > 0$ and $\beta \in (0, \alpha)$ such that $u \in C^{1,\beta}(\overline{\Omega_{\delta}}; \mathbb{R}^m)$, where $\delta := \delta(n,m, \Lip(u), \partial \Omega, \alpha, \|\phi\|_{C^{1,\alpha}(\partial \Omega)})$ and $\beta := \beta(n,\Lip u, \alpha)$.
\end{theorem}
\begin{remark}
    When $\partial \Omega$ and $\phi$ are $C^{2,\alpha}$, we can conclude that $u \in C^{2,\alpha}(\overline{\Omega_\delta};\mathbb{R}^m)$ (Corollary \ref{cor: 2alpha}).
\end{remark}
In contrast with the classical boundary regularity theorem for stationary minimal graphs (e.g. \cite[Theorem 2.3]{LO1}), our argument requires neither stationarity nor a small-excess or density hypothesis at the graph boundary, and does not assume convexity of \(\Omega\).

\begin{remark}
   Since stationary solutions are viscosity solutions \cite[p. 3, para. 1]{Sa18}, all of these results apply in particular to Lipschitz stationary solutions of the MSS. What is striking is that their proofs use only the viscosity barrier formulation and not stationarity itself. Thus, these conclusions are obtained without using the variational structure of the MSS, and the arguments are independent of the codimension. Using Theorem \ref{thm: boundary}, we prove a classification theorem for $\sing(u)$ when $u$ is a 4D stationary solution (Proposition \ref{prop: sing(u)}). 
\end{remark}

The paper is organized as follows. In Section 2, we summarize the properties of viscosity solutions to the MSS that we will need, and introduce $(n,0)$-sets focusing on the work of Santilli in \cite{San20}. Section 3 is dedicated to the proof of the Decoupling Lemma (restated more carefully in Lemma \ref{lem: decoupling2}) and its consequences. In Section 4, we briefly prove an $\varepsilon$-regularity theorem (Theorem \ref{thm: epsilonreg}) which generalizes \cite[Proposition 4.20]{BD2}. Intuitively, the theorem says that graphs of Lipschitz viscosity solutions lying close in a suitable sense to an $n$-dimensional codimension one graph embedded in $\mathbb{R}^{n+m}$ are regular. Section 5 surveys open problems for Lipschitz viscosity solutions and, as an example, we prove a Bernstein-type theorem for 3D entire solutions.

\begin{ack*}
    The author thanks Aranya Sen for their interest and feedback on earlier drafts of this manuscript --- especially with regard to the fully nonlinear PDE theory. The author also thanks Connor Mooney for pointing out the similarities between the proof of the Decoupling Lemma, including Santilli's work, and the classical ABP estimate.
\end{ack*}
\begin{ai*}
    The author used ChatGPT Edu (GPT--6.1 Sol), provided by the University of Notre Dame, to explore proof strategies; particularly for Lemma \ref{lem: decoupling} and to understand the relationship between viscosity graphs and $(n,0)$-sets. ChatGPT suggested the framework of White and Santilli after being given access to the author's work in \cite{BD2} and prompted with a general strategy (i.e. shift contact points to regular points). ChatGPT was also used for computational checks and literature searches. All AI-assisted proofs have been independently verified and refined by the author, and this manuscript contains no AI-generated text.
\end{ai*}

\section{Preliminaries}
Suppose $\Omega \subset \mathbb{R}^n$ is a bounded domain with Lipschitz boundary, unless otherwise stated.

\subsection{Viscosity solutions}
We record a few properties of viscosity solutions to the MSS which we will use in the paper. The most important is Savin's Allard Theorem.
\begin{theorem}[Allard Theorem]\label{Allard}
    Suppose $u \in C^0(\overline{B_1^n(0)};\mathbb{R}^m)$ solves the MSS in $B_1^n(0)$ in the viscosity sense and satisfies
    \[\|u - \ell\|_{L^\infty(B_1^n(0)} \leq \varepsilon,\]
    for some affine map $\ell: \mathbb{R}^n \rightarrow \mathbb{R}^m$, $\ell(x) := b + Ax$. Then there is an $\varepsilon_\ast := \varepsilon_\ast(n,m,|A|) > 0$ such that, if $\varepsilon < \varepsilon_\ast$, then $u \in C^2(B_{1/2}^n(0); \mathbb{R}^m)$ with
    \[\|D^2u\|_{C^\alpha(B_{1/2}^n(0))} \leq C\varepsilon\]
    for some constant $C$ independent of $\varepsilon$.
\end{theorem}
\begin{remark}\label{rmk:analytic}
    Any $C^2$ viscosity solution is a classical solution (and vice versa) \cite[Proposition 1.3]{Sa18}. Hence, $C^2$ viscosity solutions are analytic by Morrey's Theorem (e.g. \cite[Theorem 2.2]{LO1}).
\end{remark}
In contrast to the Allard Theorem for stationary varifolds, Savin's theorem only needs $L^\infty$ closeness to an affine function. This makes it particularly well suited to studying regularity problems for minimal graphs. We will also make use of the stability property of viscosity solutions.
\begin{proposition}[Stability; \cite{Sa18}, Proposition 1.4]\label{prop: stability}
If $\{u_k\}_1^\infty \subset C^0(\overline{\Omega}; \mathbb{R}^m)$ is a sequence of viscosity solutions for the MSS and $u_k \rightarrow u$ uniformly on compact subsets of $\Omega$, then u is a viscosity solution.
\end{proposition}

From Theorem \ref{Allard}, a Taylor expansion near any point of differentiability shows that a viscosity solution is analytic in any neighborhood of a point of differentiability. Thus, Lipschitz viscosity solutions are analytic at $\mathcal{H}^n$-a.e. interior points by Rademacher's Theorem. Furthermore, If $\sing(u)$ represents the set of interior points at which $u$ fails to be analytic (i.e. its \emph{singular set}), then $\sing(u)$ is relatively closed in the domain of $u$. Analogously, the set of interior regular points for $u$, defined as the relative complement of $\sing(u)$ in its domain, is open.

\subsection{$(n,0)$-sets}
We define $(n,0)$-sets and state the results of Santilli relevant to this paper. For a thorough treatment, see their original work \cite{San20} (see also \cite{Wh16}).

\begin{definition}[An $(n,0)$-set]\label{def: (n,0)}
A relatively closed set $S$ in an open set $U \subset\R^{n+m}$ is called an \emph{$(n,0)$-set}
if, for every $H \in C^2(\mathbb{R}^{n+m})$ and any local maximum $p_0$ of $H\lvert_S$ such that $\nabla H(p_0) \neq 0$, we have
\[
 \mu_1(p_0)+\cdots+\mu_n(p_0)\le0
\]
where $\mu_i(p_0)$ are the $n$ smallest eigenvalues of $D^2H(p_0)$ ordered least to greatest. 
\end{definition}
\begin{remark}
    Notice that, similarly to Savin's viscosity solutions, $(n,0)$-sets are defined with the barrier principle in mind. 
\end{remark}
An important property of $(n,0)$-sets that we will take advantage of is that they have a well-defined notion of normal bundle.

\begin{definition}[Generalized Unit Normal Bundle]
For a closed set $S \subset\R^{n+m}$, its \emph{generalized unit normal bundle} is
\begin{equation}\label{eq:normal}
 N(S)=\{(p,\nu)\in S\times \mathbb{S}^{n+m-1}:
       \dist(p+t\nu,S)=t\text{ for some }t>0\}.
\end{equation}
In other words, the generalized unit normal bundle consists of the ordered pairs $(p,\nu) \in S \times \mathbb{S}^{n+m-1}$ such that $B_{t}^{n+m}(p + t\nu) \cap S = \{p\}$ for some $t > 0$. Notice that the definition is independent of whether $T_pS$ is well-defined, allowing us to define unit normal vectors at singular points for $S$. We call $\nu$ a \emph{proximal unit normal} for $S$. If $Z\subset S$, we
write $N(S)|_Z=\{(p,\nu)\in N(S):p\in Z\}$.
\end{definition}

The next theorem says that if $Z$ is $\mathcal{H}^n$-null, then the set of proximal unit normals to $Z$ is $\mathcal{H}^{n+m-1}$-null in the unit sphere.\footnote{Santilli proves a more general statement, of which Theorem \ref{thm: santilli} is a special case. However, we will only use the theorem in the form stated above.} 
\begin{theorem}[\cite{San20}, Theorem 3.8]\label{thm: santilli}
Let $S$ be a relatively closed $(n,0)$-set in an open set $U \subset\R^{n+m}$,
and suppose $S$ is a countable union of sets of finite $\mathcal{H}^n$-measure.
Let $K=\overline S$, with closure taken in $\R^{n+m}$. If $Z\subset S$ is a Borel set and
$\mathcal{H}^n(Z)=0$, then
\begin{equation*}\label{eq: lusin}
 \mathcal{H}^{n+m-1}\bigl(N(K)|_Z\bigr)=0.
\end{equation*}
\end{theorem}
\begin{remark}
    In the proof of Lemma \ref{lem: decoupling2}, we will take $S = \mathscr{G}_u$ and $Z = \psi_u(\sing(u))$ for a given Lipschitz viscosity solution $u$, where $\psi_u$ is the graph map. 
\end{remark}

Let $G(n, \mathbb{R}^{n+m})$ represent the Grassmannian of $n$-planes passing through the origin in $\mathbb{R}^{n+m}$. In the proof of the next proposition, we will need the following variational identity: if $H \in C^2(\mathbb{R}^{n+m})$, then
\begin{equation}
   \min_{T \in G(n,\mathbb{R}^{n+m})} \tr_T D^2H(p) = \mu_1(p) + \cdots + \mu_n(p). \label{eq: trace}
\end{equation}
\begin{proposition}\label{prop: euivalent}
     A function $u \in C^0(\overline \Omega;\mathbb{R}^m)$ solves the MSS in $\Omega$ in the viscosity sense if and only if  $\mathscr{G}_u$ is an $(n,0)$-set in $\mathbb{R}^{n+m}$.
\end{proposition}
\begin{proof}
    We prove each implication separately.

   \medskip\noindent$\Rightarrow$: Relative closedness is immediate. Suppose $H \in C^2(\mathbb{R}^{n+m})$ is such that $H \lvert_{\mathscr{G}_u}$ has a local maximum at $p_0 := (x_0, u(x_0))$ and $\nabla H(p_0) \neq 0$. Set $m := H(p_0)$ and suppose $\sum_1^n \mu_i(p_0) > 0$. Without loss of generality, we may assume $m=0$. We claim that $H$ is a comparison function.
    
    Indeed, the variational identity \eqref{eq: trace} implies that
    \[0 < \tr_T D^2H(p_0) = \Delta_T H(p_0)\]
    whenever $T \in G(n,\mathbb{R}^{n+m})$ satisfies $T \perp \nabla H(p_0)$. Since $H \in C^2$, this identity holds in a neighborhood $U$ of $p_0$ in $\mathbb{R}^{n+m}$. Moreover, $H(x,u(x)) \leq 0$ for each $x \in \Omega$ near $x_0$ and $H((x_0,u(x_0)) = 0$. Hence, $\mathscr{G}_u$ is touched by $\{H = 0\}$ from the side $\{H \leq 0\}$. This contradicts that $u$ is a viscosity solution to the MSS.
    
    \medskip\noindent$\Leftarrow$: Suppose now that $\mathscr{G}_u$ is an $(n,0)$-set in $\mathbb{R}^{n+m}$. Let $H \in C^2(\mathbb{R}^{n+m})$ be such that $\mathscr{G}_u \subset \{H \leq 0\}$ with $H(x_0,u(x_0)) = 0$ for some $x_0 \in \Omega$. Set $p_0 := (x_0, u(x_0))$. Without loss of generality, we may assume $\nabla H(p_0) \neq 0$. We need to find an $n$-plane $T \perp \nabla H(p_0)$ such that 
    \[\tr_T D^2 H(p_0) \leq 0.\]
    
    For each $t > 0$, define
    \[H_t(x,z) = \frac{e^{tH(x,z)} - 1}{t}.\]
    Since the real function
    \[s \mapsto \frac{e^{ts}-1}{t}\]
    is strictly increasing in $s \in \mathbb{R}$, we know that for each $t$ we have $\mathscr{G}_u \subset \{H_t \leq 0\}$ with $H_t(p_0) = 0$. Direct computation yields
    \[\nabla H_t(p_0) = \nabla H(p_0), \quad D^2 H_t(p_0) = D^2H(p_0) + t \nabla H(p_0) \otimes \nabla H(p_0).\]
    Since $\mathscr{G}_u$ is an $(n,0)$-set, we know that 
    \[\mu_1^t(p_0) + \cdots + \mu_n^t(p_0) \leq 0 \text{ for each } t,\]
    where the $\mu_i^t(p_0)$ are the $n$-smallest eigenvalues of $D^2H_t(p_0)$ ordered least to greatest. 
    
    For each $t$, let $\{e_1^t,\ldots, e_n^t\}$ be an orthonormal basis of eigenvectors for $D^2H_t(p_0)$ corresponding to the eigenvalues $\mu_1^t(p_0),\ldots, \mu_n^t(p_0)$. Define the $n$-plane $T_t := \Span\{e_1^t,\ldots, e_n^t\} \in G(n,\mathbb{R}^{n+m})$. Then 
    \[\tr_{T_t} D^2H(p_0) + t|\pi_{t} \nabla H(p_0)|^2 \leq \tr_{T_t} D^2H_t(p_0) \leq 0 \text{ for each } t,\]
    where $\pi_t$ denotes orthogonal projection onto $T_t$. On the other hand,
    \[\tr_{T_t} D^2H(p_0) \geq -n\|D^2H(p_0)\|_\infty\]
    implying 
    \[|\pi_t \nabla H(p_0)|^2 \leq \frac{n\|D^2H(p_0)\|_\infty}{t} \rightarrow 0 \text{ as } t \rightarrow +\infty.\]
    By compactness of the Grassmannian, there is a sequence $t_j \rightarrow \infty$ and an $n$-plane $T \in G(n,\mathbb{R}^{n+m})$ such that $\|\pi_{t_j} - \pi_T\| \rightarrow 0$ as $j \rightarrow \infty$, where $\|\cdot \|$ is the operator norm and $\pi_T$ is orthogonal projection onto $T$. In particular, we have $\pi_T \nabla H(p_0) = 0$ so 
    \[\tr_{t_j} D^2H(p_0) = \tr(\pi_{t_j} D^2H(p_0)) \rightarrow \tr(\pi_T D^2H(p_0)) = \tr_T D^2 H(p_0) \leq 0.\]
    Since the condition $\pi_T\nabla H(p_0) = 0$ implies $T \perp \nabla H(p_0)$, $u$ is a viscosity solution. 
\end{proof}

While a relationship between $(n,0)$-graphs and viscosity graphs has been observed (e.g. see \cite{San20}), it is unclear if a proof regarding their equivalency has been previously recorded.

\section{The Decoupling Lemma and its consequences}
This section contains the proofs of the Decoupling Lemma and its consequences listed in the introduction. To state the lemma precisely, we need to introduce the \emph{Pucci extremal operators} (see \cite{Mo19} for a helpful exposition). 

Let $0 < \lambda \leq \Lambda < \infty$. For a matrix $X\in\Sym(n)$, we define 
\begin{equation}\label{eq:pucci}
 \mathcal{M}^+_{\lambda,\Lambda}(X)=\sup_{\lambda I\le A\le I \Lambda}\tr(AX),
 \qquad
 \mathcal{M}^-_{\lambda,\Lambda}(X)=\inf_{\lambda I\le A\le I \Lambda}\tr(AX).
\end{equation}
Here, the matrices $A$ in the supremum and infimum are assumed symmetric and the notation $A \geq B$ for $A,B \in \Sym(n)$ means $A - B$ is positive semi-definite. Equivalently, if $\mu_i$ are the (real) eigenvalues of $X$, then
\begin{equation}
 \mathcal{M}^+_{\lambda,\Lambda}(X)=\Lambda\sum_{\mu_i>0}\mu_i+
 \lambda\sum_{\mu_i<0}\mu_i,
 \qquad
 \mathcal{M}^-_{\lambda,\Lambda}(X)=\lambda\sum_{\mu_i>0}\mu_i+
 \Lambda\sum_{\mu_i<0}\mu_i.
\end{equation}
The operator $\mathcal{M}_{\lambda,\Lambda}^+$ is called the \emph{maximal Pucci operator}, and $\mathcal{M}_{\lambda,\Lambda}^-$ is called the \emph{minimal Pucci operator}.

Let $A(x) := (a^{ij}(x))_{i,j = 1,\ldots,n} \in \Sym(n)$ be a matrix of measurable functions defined on a Lipschitz domain $\Omega$ such that $\lambda I \leq A \leq \Lambda I$ in $\Omega$. Observe that, from \eqref{eq:pucci}, if $u \in C^2(\Omega)$ solves 
\begin{equation}\label{eq: homolinear}
    a^{ij}(x)u_{x_i x_j}(x) = 0 \text{ in } \Omega
\end{equation}
then $u$ satisfies the \emph{homogeneous Pucci inequalities}
\begin{equation}\label{eq: homoPucci}
    \mathcal{M}_{\lambda,\Lambda}^-(D^2u(x)) \leq 0 \leq \mathcal{M}_{\lambda,\Lambda}^+(D^2u(x)) \text{ in } \Omega.
\end{equation}
In the converse direction, if $u \in C^2(\Omega)$ satisfies \eqref{eq: homoPucci} then $u$ solves \eqref{eq: homolinear} for \emph{some} measurable coefficient matrix $A$ with eigenvalues in $[\lambda, \Lambda]$. 

For $u \in C^0(\Omega)$, we say that $\varphi \in C^2(\Omega)$ \emph{touches $u$ from above} at $x_0$ if $\varphi(x) - u(x) \geq 0$ in $\Omega$ with equality at $x_0$. Likewise, we say that $\varphi$ \emph{touches $u$ from below} at $x_0$ if instead $\varphi(x) - u(x) \leq 0$ in $\Omega$. Due to the preceding discussion, we define a weak notion of solution by transferring derivatives to $C^2$ functions $\varphi$ which touch $u$ from above or below at points in $\Omega$, much like what was done for the system. Indeed, we say that $u$ is a \emph{viscosity solution} in $\Omega$ if the following two conditions hold:
\begin{itemize}
    \item[(i)] If $\varphi \in C^2(\Omega)$ touches $u$ from above at $x_0$, then $\mathcal{M}_{\lambda,\Lambda}^+(D^2 \varphi(x_0)) \ge 0$.
    \item[(ii)] If $\varphi \in C^2(\Omega)$ touches $u$ from below at $x_0$, then $\mathcal{M}_{\lambda, \Lambda}^-(D^2\varphi(x_0)) \le 0$.
\end{itemize}
When $u$ satisfies (i) $u$ is called a \emph{viscosity subsolution}, and when $u$ satisfies (ii) $u$ is called a \emph{viscosity supersolution}. Loosely speaking, the set of all viscosity solutions in $\Omega$ can be viewed as the collection of all solutions to equations \eqref{eq: homolinear}. We will show that each component function of a Lipschitz viscosity solution to the MSS satisfies the homogeneous Pucci inequalities. Hence, component functions $u^\beta$ of Lipschitz viscosity solutions formally solve \eqref{eq: homolinear} for matrices $A_\beta = (a_\beta^{ij})$.

\begin{lemma}\label{lem: decoupling2}
    Suppose $u \in \Lip(\Omega; \mathbb{R}^m) \cap C^0(\overline{\Omega};\mathbb{R}^m)$ solves the MSS in $\Omega$ in the viscosity sense and suppose $\Lip u = M$. Then for each $\beta \in \{1,\ldots,m\}$, the component functions $u^\beta$ of $u$ satisfy the homogeneous Pucci inequalities \eqref{eq: homoPucci} for $\lambda_M := (1 + M^2)^{-1}$ and $\Lambda = 1$.
\end{lemma}
\begin{remark}
We note that, since the MSS is invariant under the transformation $u \mapsto -u$ and the Pucci operators are positively 1-homogeneous, it suffices to show that each $u^\beta$ cannot be touched from above  at a point $x_0 \in \Omega$ by $\varphi \in C^2$ satisfying
    \begin{equation}
        \mathcal{M}_{\lambda_M, 1}^+(D^2\varphi(x_0)) < 0. \label{eq: pucci<}
    \end{equation}
\end{remark}

Lemma \ref{lem: decoupling} is natural since the component functions $u^\beta$ of a classical solution $u \in C^2(\Omega;\mathbb{R}^m)$ solve $g^{ij}(x) u_{x_i x_j}^\beta(x) = 0$ in $\Omega$ in the viscosity sense with $\lambda_MI \leq g^{-1} \leq I$. In this context, we mean that if $\varphi \in C^2$ touches $u^\beta$ from above (below) at $x_0 \in \Omega$, then 
\[g^{ij}(x_0)\varphi_{x_i x_j}(x_0) \geq 0 \quad (\leq 0).\] 
It follows that $u^\beta$ satisfies \eqref{eq: homoPucci} at $x_0$. Since Theorem \ref{thm: 2d} and Theorem \ref{thm: partialreg} follow immediately from Lemma \ref{lem: decoupling} and the existing theory, we first briefly sketch their proofs. 
\begin{proof}[Proof of Theorem \ref{thm: 2d}]
    Each component of $u$ is $C^{1,\alpha}$ due to \cite[Theorem 12.4]{GT01} and \cite{Je96}, so $u$ is as well. Theorem \ref{Allard} applied after Taylor expansion near each point of $\Omega$ gives full interior regularity.
\end{proof}
The proof above derives regularity of 2D solutions directly from Lemma \ref{lem: decoupling} and the analyticity of 2D solutions to scalar equations. Since regularity is not guaranteed on higher dimensional domains, new techniques are needed to understand the general case. Nevertheless, we have quantitative partial regularity independent of $n$ and $m$.

\begin{proof}[Proof of Theorem \ref{thm: partialreg}]
    The theorem is proved in \cite[Theorem 4.12]{BD2} assuming Lemma \ref{lem: decoupling2} holds. Its proof is a variant of Armstrong, Silvestre, and Smart's partial regularity theorem for scalar equations proved in \cite{ASS12}; however, it uses Theorem \ref{Allard}. The estimate for $\varepsilon$ follows from the lower bound on $\varepsilon$ in the $W^{2,\varepsilon}$ estimate for fully nonlinear scalar equations \cite[Theorem 1]{NT23} (see \cite{Mo19} for the original $\varepsilon$ bound). 
\end{proof}

It is not hard to see that viscosity graphs are invariant under rigid motions, since the same transformation applied to the comparison surface yields a comparison surface for the transformed graph. Recall that $w(x) = L(x,u(x))$ for $L : \mathbb{R}^n \times \mathbb{R}^m \rightarrow \mathbb{R}$ given by
\[L(x,z) := a \cdot x + b \cdot z + c.\]
\begin{proof}[Proof of Proposition \ref{prop: harnack}]
   If $b = 0$, $w$ is affine and the conclusion is immediate. Suppose $b \neq 0$. If $R \in O(\mathbb{R}^{m})$ is a rotation, then $Ru: \Omega \rightarrow \mathbb{R}^{m}$ is a viscosity solution also. Taking the first row of $R$ to be $b/|b|$, we see that $(Ru)^1 = b/|b| \cdot u$ is a Lipschitz viscosity solution to a uniformly elliptic scalar PDE by Lemma \ref{lem: decoupling2}. Applying positive homogeneity of the scalar Pucci inequalities, we see that $b \cdot u$ satisfies the Pucci inequalities as well. Since $D^2 w = D^2(b\cdot u)$, this extends to $w$ so we may apply the scalar Harnack inequality to $w$ (e.g. see \cite[Chapter 4]{CC95}).
\end{proof}

Proposition \ref{prop: harnack} implies a strong maximum principle and half-space Liouville theorem for Lipschitz viscosity solutions. In the following corollaries, $w$ represents the function defined in Proposition \ref{prop: harnack}
\begin{corollary}[Strong Maximum Principle]\label{cor: max}
    Suppose $\Omega \subset \mathbb{R}^n$ is connected and $u \in \Lip(\Omega; \mathbb{R}^m) \cap C^0(\overline \Omega; \mathbb{R}^m)$ solves the MSS in $\Omega$ in the viscosity sense. If $w \geq 0$ on $\overline{\Omega}$ with $w(x_0) = 0$ for some $x_0 \in \Omega$, then $w \equiv 0$ in $\Omega$. Equivalently, if $\mathscr{G}_u$ touches a hyperplane from one side, it lies entirely in the hyperplane. 
\end{corollary}
\begin{proof}
    Suppose $w(x_0) = 0$. The Harnack inequality applied in $B_{2r}^n(x_0) \subset \Omega$ implies
    \[\sup_{B_r^n(x_0)} w \leq C \inf_{B_r^n(x_0)} w = 0.\]
    That is, $w \equiv 0$ on $B_r^n(x_0)$. A standard covering argument implies $w \equiv 0$ in $\Omega$.
\end{proof}

\begin{corollary}[Half-space Liouville Theorem]\label{cor: liouville}
    Suppose $u \in \Lip(\mathbb{R}^n; \mathbb{R}^m)$ is an entire viscosity solution to the MSS on $\mathbb{R}^n$. If $\inf_{\mathbb{R}^n} w(x) > - \infty$, then $w$ is constant. Equivalently, if $\mathscr{G}_u$ lies in a closed half-space, then $\mathscr{G}_u$ lies in a hyperplane parallel to the boundary hyperplane of the half-space.
\end{corollary}
\begin{proof}
   Without loss of generality, we may assume $\inf_{\mathbb{R}^n} w(x) = 0$. Fix $x_0 \in \mathbb{R}^n$. Then the Harnack inequality says that for any $r > 0$
   \[0 \leq \sup_{B_r^n(x_0)} w \leq C(n,M) \inf_{B_r^n(x_0)} w \rightarrow 0 \text{ as } r \rightarrow +\infty.\]
   Since the supremum in the middle term is increasing in $r$, we have $w\equiv 0$.
\end{proof}

We are ready to prove the Decoupling Lemma.
\subsection{Proof of Lemma \ref{lem: decoupling}}
First, we show that if $H \in C^2(\mathbb{R}^{n+m})$ and an affine perturbation of $H$ has a local maximum at an interior point $p_0 \in \mathscr{G}_u$, we can extract a proximal unit normal. 
\begin{lemma}\label{lem:proximal}
Suppose $H\in C^2(\R^{n+m})$ and $a\in\mathbb{R}^{n+m}$.
If $p_0 \in \Int \mathscr{G}_u$ maximizes $H_a(p):=H(p)+a\cdot p$ over $\mathscr{G}_u$ and
$\nabla H_a(p_0) \ne0$, then
\[
 \left(p_0,\nu_a \right)\in N(\mathscr{G}_u)
\text{ for } \nu_a := \nabla H_a(p_0)/|\nabla H_a(p_0)|.\]
\end{lemma}
\begin{proof}
    Since $p_0$ is an interior point, there is an $r>0$ such that $B_r := B_r^{n+m}(p_0) \cap \mathscr{G}_u \subset \subset \Int \mathscr{G}_u$. Fix $p \in B_r$ and set $C := \|D^2 H\|_{L^\infty(\overline{B_r^n(p_0)})}$. Taylor's Theorem applied to $H_a$ along the segment joining $p_0$ to $p$ in $B_r^{n+m}(p_0)$ gives 
    \[0 \geq \nabla H_a(p_0) \cdot (p-p_0) - \frac{C}{2}|p-p_0|^2.\]
    If $t < \min\{|\nabla H_a(p_0)|/2C, r/2\}$, then
    \begin{align*}
        |p - (p_0 + t\nu_a)|^2 &= |p -p_0|^2 -2t\nu_a \cdot (p-p_0) + t^2 \\
        &\geq \Big(1 - \frac{Ct}{|\nabla H_a(p_0)|}\Big)|p-p_0|^2 +t^2  \\
        &\geq t^2.
    \end{align*} 
    It follows that the closed $(n+m)$-ball of radius $t$ centered at $p_0 + t\nu_a$ intersects $\mathscr{G}_u$ only at $p_0$.
\end{proof}

For the remainder of this section, the map $H_a : \mathbb{R}^{n+m} \rightarrow \mathbb{R}$ will denote the map defined in Lemma \ref{lem:proximal}. The next lemma allows us to shift contact to regular points by tilting an initial hypersurface determined by a level set of $H$ (i.e. $H \mapsto H_a$). In fact, we will show that generic tiltings of a barrier hypersurface make contact with $\mathscr{G}_u$ at regular points.

\begin{lemma}\label{lem: tilts}
Suppose $Z\subset \Int \mathscr{G}_u$ is relatively closed and that
$\mathcal{H}^{n+m-1}(N(\mathscr{G}_u)|_Z)=0$. Let $H\in C^2(\mathbb{R}^{n+m})$ and $\rho>0$. Then for Lebesgue a.e. $a \in B_\rho^{n+m}(0)$, every maximum $p_0 \in \Int \mathscr{G}_u$ of $H_a$ satisfying $\nabla H_a(p_0) \neq0$ lies in $\Int \mathscr{G}_u \setminus Z$. 
\end{lemma}
\begin{proof}
    Define
\[
 C_\rho :=1+\sup_{p\in \mathscr{G}_u}|\nabla H(p)|+\rho.
\]
If $a \in B_\rho^{n+m}(0)$ admits a maximum at $p_0\in Z$ with $\nabla H_a(p_0) \neq 0$,
Lemma~\ref{lem:proximal} gives $(p_0,\nu_a)\in N(\mathscr{G}_u)\lvert_Z$ where $\nu_a$ is defined from $H_a$ as before. Set $t_a := |\nabla H_a(p_0)|$. Then $t_a \in (0,C_\rho)$ and 
\begin{equation}
 a=t_a\nu_a-\nabla H(p_0). \label{eq: a}
\end{equation}

Set $E:= N(\mathscr{G}_u)\lvert_Z$ and consider the map
\[
 \Psi:E \times [0,C_\rho]\longrightarrow\R^{n+m},
 \qquad \Psi(p,\nu,t)=t\nu-\nabla H(p).
\]
Since $\nabla H$ is $C^1$, we may set $C_0 := \|\nabla H\|_{L^\infty(\mathscr{G}_u)} < \infty$. Then for $(p,\nu,t),(q,\mu,s) \in E$, we have 
\[|\Psi(p,\nu,t) - \Psi(q,\mu,s)| \leq C_\rho|\nu - \mu| + |t-s| + C_0|p-q| \leq C_1|(p,\nu,t) - (q,\mu,s)|.\]
That is, $\Psi$ is Lipschitz on $E$. Since $\mathcal{H}^{n+m-1}(E) = 0$, we have $\mathcal{H}^{n+m}(E \times [0,C_\rho]) = 0$ also. Using that $\Psi$ is Lipschitz, we conclude
\[\mathcal{L}^{n+m}(\Psi(E \times [0,C_\rho])) = 0\]
where $\mathcal{L}^{n+m}$ is the Lebesgue measure on $\mathbb{R}^{n+m}$. Since this image contains those $a \in B_\rho^{n+m}(0)$ admitting a maximum $p_0 \in Z$ such that $\nabla H_a(p_0) \neq 0$ by \eqref{eq: a}, the set of such $a$ is $\mathcal{L}^{n+m}$-null.
\end{proof}

Suppose now that $u^\beta$ is touched from above at $x_0 \in \Omega$ by $\varphi \in C^2(\Omega)$ satisfying \eqref{eq: pucci<} with $\Lambda = 1$ and $\lambda = \lambda_M$. Without loss of generality, we may assume $\beta = 1$. After localizing and replacing $\varphi$ by $\varphi(x) + |x-x_0|^4$ (which we do not relabel), we may assume without loss of generality that $\varphi \geq u^1 + |x-x_0|^4$ in $\Omega \setminus \{x_0\}$. Define $H \in C^2(\Omega \times \mathbb{R}^m)$ by $H(x,z) := z^1 - \varphi(x)$. Let $\psi_u$ be the graph map for $u$. Then $H \lvert_{\mathscr{G}_u} = H \circ \psi_u \leq 0$ in $\Omega$ with equality at $x_0$. We split the remainder of the proof into three steps.

\medskip
\noindent \emph{Step 1: Small tiltings have nearby maxima.} For $a \in B_\rho^{n+m}(0)$, set
\[H_a(x,z) := H(x,z) + a \cdot (x,z), \qquad a := (a^x, a^z) \in \mathbb{R}^n\times \mathbb{R}^m\]
and let $Z := \psi_u(\sing(u))$. Choose $0 < r < 1$ small enough that $B_r^n(x_0) \subset \subset \Omega$ and set $h_a := H_a \lvert_{\mathscr{G}_u}$. Note that
\[h_a(x) = (H \circ \psi_u)(x) + a \cdot \psi_u(x)\]
implying that for $x$ satisfying $|x-x_0| = r$ we have 
\[h_a(x) - h_a(x_0) \leq -r^4+|a|\sqrt{1+M^2}r.\]
Taking $\rho$ small depending on $M$ and $r$ gives 
\[h_a(x) - h_a(x_0) \leq -r^4/2.\]
Continuity of $h_a$ and compactness of $\overline{B_r^n(x_0)}$ imply that $h_a$ has a maximum on $\overline{B_r^n(x_0)}$, while the previous estimate says this maximum must be achieved at an interior point. Equivalently, $H_a$ achieves its maximum on $\mathscr{G}_u$ in $\psi_u(B_r^n(x_0))$.

\medskip
\noindent \emph{Step 2: Generic tiltings achieve their maximum at regular points.} 
Let $\{e_\beta\}_{\beta = 1}^m$ denote the standard basis for $\mathbb{R}^m$. At $(x,z) \in \mathbb{R}^n \times \mathbb{R}^m$ we have 
\[\nabla H_a(x,z) = (-\nabla_x \varphi + a^x, e_1 + a^z) \neq 0\]
since $\rho < r < 1$ and $a \in B_\rho^{n+m}(0)$. Now, $\mathcal{H}^n(Z) = 0$ so that $\mathcal{H}^{n+m-1}(N(\mathscr{G}_u)\lvert_Z) = 0$ by Theorem \ref{thm: santilli}. Lemma \ref{lem: tilts} says that the tiltings $a$ for which $h_a$ has a maximum in $Z$ form a Lebesgue null set $\mathcal{N}$. Choose $a \in B_\rho^{n+m}(0) \setminus \mathcal{N}$ and let $x_a$ be a maximum point for $h_a$. The previous step shows that $x_a \in \Int \mathscr{G}_u \setminus Z$. 

\medskip
\noindent\emph{Step 3: Contradict barrier principle.} Let $p_a := (x_a,u(x_a))$. By construction, $H_a|_{\mathscr{G}_u}$ has a local maximum at $p_a$. Since $x_a$ is a regular point, $\mathscr{G}_u$ is a smooth minimal submanifold in a neighborhood of $p_a$, and
\[
\nabla H_a(p_a)\perp T_{p_a}\mathscr{G}_u.
\]
Set $T_a := T_{p_a}\mathscr{G}_u$. Then the barrier principle for smooth minimal submanifolds implies
\[
    \tr_{T_a}D^2H_a(p_a) = -\tr(g^{-1}(x_a)D^2\varphi(x_a)) \leq 0.
\]
Since $\lambda_M I\leq g^{-1}(x_a)\leq I$, the definition of the maximal Pucci operator gives
\[
    -\tr(g^{-1}(x_a)D^2\varphi(x_a))
\geq
-\mathcal M_{\lambda_M,1}^+(D^2\varphi(x_a))
>0,
\]
contradicting the barrier principle. \hfill $\square$

\subsection{Boundary regularity}
When paired with the regularity theory for scalar viscosity solutions, Lemma \ref{lem: decoupling} leads to a proof of Theorem \ref{thm: boundary}. This is because sufficient regularity at the boundary of a Lipschitz viscosity graph implies quantitative flatness at $\partial \Omega$ due to \cite[Theorem 1.1]{SS14}, allowing us to apply Savin's Allard Theorem at interior points near the boundary. A standard two point reduction gives $C^{1,\beta}$ regularity up to the boundary for some $\beta \in (0, \alpha)$ when the boundary data of the solution are $C^{1,\alpha}$ for some $\alpha \in (0,1)$. 

Let $\Omega \subset \mathbb{R}^n$ be a bounded $C^2$ domain. Recall that, for $\delta > 0$, we set  
\[\Omega_\delta := \{x \in \Omega : \dist(x, \partial \Omega) < \delta\}.\]
\begin{proof}[Proof of Theorem \ref{thm: boundary}]
The boundary regularity theorem of Silvestre--Sirakov \cite[Theorem 1.1]{SS14} shows that there are constants $\beta \in (0,\alpha)$ and $r > 0$ and a map $A \in C^\beta(\partial \Omega; \mathbb{R}^{m\times n})$ such that, if $p \in \partial \Omega$ is fixed and we define
\[\ell_p(y) := u(p) + A(p)(y-p),\]
then
\begin{equation}\label{est2}
    |u(y) - \ell_p(y)| \leq C|y-p|^{1 + \beta} \text{ for all } y \in \Omega_r.
\end{equation}
Furthermore, 
\begin{equation}\label{est3}
    |A(p) - A(q)| \leq C|p-q|^\beta \text{ for all } p,q \in \partial \Omega,
\end{equation}
where the constant $C$ is uniform due to the compactness of $\partial \Omega$.

Fix $x \in \Omega_r$. Set $d:= \dist(x,\partial \Omega)$ and set $\rho := d/2$. Choose $p \in \partial \Omega$ such that $|x-p| = d$. Then $B_\rho^n(x) \subset \Omega$. The triangle inequality shows that for each $y \in B_\rho^n(x)$ we have $|y-p| \leq 3\rho$. If $3\rho < r$, then 
\begin{equation}
    \|u - \ell\|_{L^\infty(B_\rho^n(x))} \leq C\rho^{1 + \beta}
\end{equation}
by \eqref{est2}. We now re-scale to the unit ball: for $z \in B_1^n(0)$, we define
\begin{align*}
    v(z) &:= v_{x,\rho}(z) = \rho^{-1}(u(x + \rho z) - u(x)) \\
    L(z) &:= L_{x,\rho}(x) = \rho^{-1}(\ell(x + \rho z) - u(x)).
\end{align*}
Notice that $DL(z) = A(p)$ is uniformly bounded. Then 
\begin{equation*}
    \|v - L\|_{L^\infty(B_1^n(0))} \leq C \rho^\beta.
\end{equation*}
Choose $\delta > 0$ so small that $d < \delta$ implies $3\rho < r$ and $C\rho^\beta < \varepsilon_\ast$, where $\varepsilon_\ast$ is the constant in Theorem \ref{Allard}. Thus, the theorem implies that $v$ is analytic in $B_{1/2}^n(0)$ and 
\begin{equation*}
    \|D^2 v\|_{L^\infty(B_{1/2}^n(0))} \leq C \rho^\beta.
\end{equation*}
Now, by a quantitative improvement of flatness estimate we have 
\begin{equation*}
    |Dv(0) - A(p)| \leq C \rho^\beta,
\end{equation*}
where $Dv(0) = Du(x)$ and $D^2v(z) = \rho D^2 u (x + \rho z)$. We have therefore shown
\begin{equation}\label{est4}
    |Du(x) - A(p)| \leq Cd^\beta, \qquad \|D^2u\|_{L^\infty(B_{d/4}^n(x))} \leq C d^{\beta-1}.
\end{equation}

The estimates derived above allow us to show that $A$ is the continuous extension of $Du$ on $\partial \Omega$. Indeed, if $q \in \partial \Omega$ and $p$ is a nearest boundary point to $x$, then
\[d \leq |x-q|, \qquad |p-q| \leq 2|x-q|.\]
Using \eqref{est3}, \eqref{est4}, the triangle inequality, and the above inequalities, we deduce 
\[|Du(x) - A(q)| \leq |x-q|^\beta.\]
Thus, $Du$ extends H\"older continuously to $\partial \Omega$ with exponent $\beta$. The estimate \eqref{est2} identifies $A(q)$ with the derivative of $u$ at $q$ relative to $\partial \Omega$.

To conclude the proof, we need to obtain uniform H\"older estimates for $Du$ in $\Omega_\delta$. Fix $x,z \in \Omega_\delta$ and set
\[h := |x-z|, \qquad d_x := d(x,\partial \Omega), \qquad d_z := d(z, \partial \Omega).\]
First, we suppose $h < d_x/4$. Then the line segment joining $x$ and $z$ lies entirely in $B_{d_x/4}^n(x)$. Using the Hessian bound in \eqref{est4}, we deduce
\begin{equation}\label{est5}
    |Du(x) - Du(z)| \leq Cd_x^{\beta-1} h = Ch^\beta\big(h/d_x\big)^{1-\beta} \leq Ch^\beta.
\end{equation}
This proves the desired estimate for $h < d_x/4$. Next, suppose $h \geq d_x/4$. Then $d_z \leq 5h$ by the triangle inequality. Let $p_x, p_z$ be nearest points on $\partial \Omega$ to $x$ and $z$, respectively. We have $|p_x - p_z| \leq 10h$ by the triangle inequality. Thus, \eqref{est3} and \eqref{est4} imply
\begin{align}
    |Du(x) - Du(z)| &\leq | Du(x) - A(p_x)| + |A(p_x) - A(p_z)| + |A(p_z) - |Du(z)| \nonumber \\
    &\leq C\big(d_x^\beta + |p_x - p_z|^\beta + d_z^\beta \big) \nonumber \\
    &\leq Ch^\beta. \label{est6}
\end{align}
Estimates \eqref{est5} and \eqref{est6} complete the proof. The constants $\beta$ and $\delta$ inherit their dependencies from our use of \cite[Theorem 1.1]{SS14} in the former case and Theorem \ref{Allard} in the latter.
\end{proof}
\begin{remark}
    Notice that Theorem \ref{thm: boundary} holds for stationary solutions to the MSS. Furthermore, we do not require $\partial \Omega$ to be strictly convex since our proof does not appeal to the theory for stationary varifolds (compare with \cite[Theorem 2.3]{LO1}).
\end{remark}

Theorem \ref{thm: boundary} along with a simple bootstrapping argument yields the following corollary. We omit the details, since it follows from boundary Schauder estimates applied component-wise to the classical MSS.
\begin{corollary}\label{cor: 2alpha}
      Assume $\Omega \subset \mathbb{R}^n$ is a bounded domain with $C^{2,\alpha}$ boundary. Suppose $u \in \Lip(\Omega; \mathbb{R}^m) \cap C^0(\overline{\Omega};\mathbb{R}^m)$ solves the MSS in $\Omega$ in the viscosity sense, and that $u \lvert_{\partial \Omega} = \phi \in C^{2,\alpha}(\partial \Omega; \mathbb{R}^m)$ for some $\alpha \in (0,1)$. Then there is a constant $\delta > 0$ such that $u \in C^{2,\alpha}(\overline{\Omega_{\delta}}; \mathbb{R}^m)$, where $\delta := \delta(n,m, \Lip(u), \alpha,  \partial \Omega, \|\phi\|_{C^{2,\alpha}(\partial \Omega)})$.
\end{corollary}

Theorem \ref{thm: boundary} allows us to characterize $\sing(u)$ for Lipschitz stationary solutions on a $C^2$ bounded domain in $\mathbb{R}^4$ with $C^{1,\alpha}$ boundary data. We say a cone $\mathcal{C}$ is \emph{regular} if it is $C^\infty$ away from its vertex.
\begin{proposition}\label{prop: sing(u)}
    Suppose $\Omega \subset \mathbb{R}^4$ is a bounded domain with $C^2$ boundary. If $u \in \Lip(\Omega; \mathbb{R}^m) \cap C^0(\overline \Omega; \mathbb{R}^m)$ is a stationary solution to the MSS in $\Omega$ such that $u \lvert_{\partial \Omega} = \phi \in C^{1,\alpha}(\partial\Omega; \mathbb{R}^m)$ for some $\alpha \in (0,1)$, then:
    \begin{itemize}
        \item[(i)] The interior singular set $\sing(u)$ is a finite subset of $\Omega$. 
        \item[(ii)] At each singularity, the tangent cone to $\mathscr{G}_u$ at the singular point is a multiplicity one regular cone and, up to rotation, is graphical. Hence, tangent cones are unique.
    \end{itemize}
\end{proposition}
\begin{proof}
    Any Lipschitz stationary solution $u$ is a viscosity solution, so we may apply Theorem \ref{thm: boundary}. Theorem \ref{thm: boundary} implies there are $\delta > 0$ and $\beta \in (0,\alpha)$ such that $u \in C^{1,\beta}(\overline{\Omega_\delta}; \mathbb{R}^m)$. Hence, $\sing(u)$ is a compact subset of $\Omega \setminus \overline{\Omega_\delta}$. Since $\sing(u)$ is discrete \cite[Theorem 3.7]{BD2}, it is finite. This proves (i).

    Fix $p \in \sing(u)$. That the tangent cones to $\mathscr{G}_u$ at $p$ are multiplicity one and graphical up to rotation is proved in \cite[Theorem 11.53 \& Theorem 11.54]{GM12}. Let $\mathcal{C}$ be a tangent cone to $\mathscr{G}_u$ at $p$. Without loss of generality, we may assume $\mathcal{C}$ has vertex at the origin. If $p_0 \in \mathcal{C} \setminus \{0\}$ is a singular point, then $\mathcal{C}$ contains a ray of singularities. However, $\mathcal{C}$ is 4D and graphical so that the Hausdorff dimension of its singular set is at most zero \cite[Theorem 3.7]{BD2}, which contradicts the ray of singular points. We conclude that $\mathcal{C}$ is regular. Since the tangent cones are multiplicity one and regular, uniqueness now follows by a famous result of Simon \cite{Si84}.
\end{proof}
\begin{remark}\label{rmk: conic}
    Condition (ii) is independent of whether $\mathscr{G}_u$ is compact. Thus, (ii) implies that 4D stationary graphs are \emph{asymptotically conic} at each singularity (see \cite[pp. 24-25]{Sm93}). 
\end{remark}
\begin{remark}
    An analogous theorem holds for 5D Lipschitz special Lagrangian graphs (as well as Remark \ref{rmk: conic}) if we assume, for instance, that $\partial \Omega \in C^{2,\alpha}$ and $\phi \in C^{2,\alpha}(\partial \Omega; \mathbb{R}^m)$. To see this, repeat the argument with \cite[Theorem 3.9]{BD2} (see also the proof) and \cite[Theorem 1.3]{SS14} replacing \cite[Theorem 3.7]{BD2} and \cite[Theorem 1.1]{SS14}, respectively.
\end{remark}

In \cite{Dim25}, it was shown that, for any fixed $N \in \mathbb{N}$, there is a strictly stable stationary solution to the MSS defined on a smooth domain $\Omega \subset \mathbb{R}^4$ having smooth boundary data and precisely $N$ isolated singularities in $\Omega$. We therefore have proved an optimal regularity statement for stationary solutions with smooth boundary data when $n = 4$:\footnote{See \cite[Remark 2]{DG26} for an analogous statement for the special Lagrangian equation.}
\vspace{.2cm}
\begin{center}
    \emph{If $\Omega \subset \mathbb{R}^4$ is a bounded smooth domain and $u \in \Lip(\Omega;\mathbb{R}^m)$ is a stationary solution to the MSS in $\Omega$ with $C^\infty$ boundary data, then singularities may occur and their number can be arbitrarily large.}
\end{center}

\section{An $\varepsilon$-regularity theorem}
We establish an $\varepsilon$-regularity theorem for Lipschitz viscosity solutions which generalizes \cite[Proposition 4.20]{BD2}. Since the result is local, we only consider $\Omega := B_1^n(0)$. We define the \emph{oscillation} of a function $u \in C^0(B_1^n(0);\mathbb{R}^m)$ by 
\[\osc_{B_1^n(0)} u:= \sup_{x,y \in B_1^n(0)}|u(x) - u(y)|.\]
\begin{theorem}[$\varepsilon$-regularity]\label{thm: epsilonreg}
    Fix $n, m \in \mathbb{N}$ and let $M > 0$. Then there is an $\varepsilon_0 := \varepsilon_0(n,m,M) > 0$ such that, if $\varepsilon < \varepsilon_0$ and $u := (u^1, \tilde{u}) \in \Lip(B_1^n(0); \mathbb{R}^m) \cap C^0(\overline{B_1^n(0)}; \mathbb{R}^m)$ solves the MSS in $B_1^n(0)$ in the viscosity sense with
    \[\Lip u \leq M, \qquad \osc_{B_1^n(0)} \tilde{u} < \varepsilon,\]
    then $u$ is analytic in $B_{1/2}^n(0)$. 
\end{theorem}

Heuristically, this says that graphs of Lipschitz viscosity solutions to the MSS lying suitably close to a codimension one graph embedded in $\mathbb{R}^{n+m}$ are regular. First, we state a simple lemma.
\begin{lemma}\label{lem: MSEsolution}
    Suppose $u \in  \Lip(B_1^n(0); \mathbb{R}^m) \cap C^0(\overline{B_1^n(0)}; \mathbb{R}^m)$ solves the MSS in $B_1^n(0)$ in the viscosity sense and $u := (v, 0) \in \mathbb{R} \times \mathbb{R}^{m-1}$, where $v : B_1^n(0)\rightarrow \mathbb{R}$. Then $v$ is a viscosity solution for the scalar MSE. In particular, $u$ is analytic in $B_1^n(0)$.
\end{lemma}
\begin{proof}[Proof Sketch]
    Suppose $\varphi \in C^2$ touches $v$ from above at $x_0$. Then $H(x,z) := z^1 - \varphi(x) + C|\tilde{z}|^2$ can be shown to be a comparison function for some constant $C$ using a modified but simpler version of the argument found in \cite[p. 21]{BD2}. We omit the technical details.
\end{proof}

Theorem \ref{thm: epsilonreg} follows from Lemma \ref{lem: MSEsolution} and the stability property.
\begin{proof}[Proof of Theorem \ref{thm: epsilonreg}]
    Suppose the conclusion of the theorem is false. Then there is a sequence $\{u_j\}_1^\infty$ of viscosity solutions, $u_j := (u_j^1,\tilde{u}_j)$, satisfying
    \[\Lip u_j \leq M,\qquad \osc_{B_1^n(0)}\tilde{u}_j \rightarrow 0\]
    while each $u_j$ has a singular point $x_j \in B_{1/2}^n(0)$. After translation, we may assume $u_j(0)= 0$ for each $j$. By the Arzel\'a--Ascoli Theorem, we have $u_j \rightarrow w := (v,0)$ locally uniformly in $B_1^n(0)$ for some $w$ with $\Lip w \leq M$. Compactness of $\overline{B_{1/2}^n(0)}$ implies we may also assume $x_j \rightarrow x_\infty \in \overline{B_{1/2}^n(0)}$. Then Proposition \ref{prop: stability} implies $w$ is a viscosity solution on $B_1^n(0)$ as well. Lemma \ref{lem: MSEsolution} then gives smoothness of the first component function $v$ of $w$; hence, $w$ is smooth also.

    We claim that $x_j$ is in fact a regular point for $u_j$ when $j$ is large. Indeed, if $\ell : \mathbb{R}^n \rightarrow \mathbb{R}^m$ is the defining affine map for the tangent plane to $\mathscr{G}_w$ at $x_\infty$, then smoothness of $w$ implies
    \[r^{-1} \sup_{B_r^n(x_\infty)}|w - \ell| \rightarrow 0 \text{ as } r\rightarrow 0.\]
    Let $\varepsilon_\ast$ be the constant in Theorem \ref{Allard}. Taking $r$ sufficiently small, we may assume
    \[r^{-1} \sup_{B_{r}^n(x_\infty)}|w - \ell| < \varepsilon_\ast/2 \Rightarrow \sup_{B_r^n(x_\infty)}|u_j - \ell| < r\varepsilon_\ast\]
    for sufficiently large $j$. Theorem \ref{Allard} ensures the $u_j$ are smooth in the ball $B_{r/2}^n(x_\infty)$ for large $j$. Since $x_j \rightarrow x_\infty$, we have reached a contradiction.
\end{proof}

\section{Discussion}
We are left with many unresolved questions concerning the regularity of Lipschitz viscosity solutions to the system, along with whether viscosity solutions are stationary. For example, while Lipschitz viscosity solutions on a domain $\Omega \subset \mathbb{R}^2$ are regular by Lemma \ref{lem: decoupling2}, the argument does not generalize to higher dimensional domains. However, due to the observed parallels with the stationary case and the fact that it holds when the codimension is one, it is natural to expect Lipschitz viscosity solutions are stationary. 
\begin{conj}\label{conj: stationary}
    Fix $n \geq 3$ and $m \geq 2$. If $u \in \Lip(B_1^n(0); \mathbb{R}^m) \cap C^0(\overline \Omega; \mathbb{R}^m)$ solves the MSS in $\Omega$ in the viscosity sense, then $u$ is a stationary solution.
\end{conj}
\begin{remark}
     Interestingly, the case $n = 2$ and $m \geq 2$ is derived from an ABP argument. For potential strategies when $n \geq 3$, including a necessary and sufficient condition for stationarity, the reader should consult \cite[Section 4]{BD2}.
\end{remark}
Conjecture \ref{conj: stationary} can be viewed as a little cousin to the Lawson-Osserman Conjecture (see \cite[Conjecture 2.1]{LO1}). Since classical solutions are stationary, Theorem \ref{thm: 2d} shows that Conjecture \ref{conj: stationary} holds when $n = 2$ mirroring the Lawson-Osserman Conjecture \cite{HMT23}. An intermediate step is to determine the regularity of 3D viscosity solutions.
\begin{conj}\label{conj: Bernstein}
    Suppose $\Omega \subset \mathbb{R}^3$. If $u \in \Lip(\Omega; \mathbb{R}^m) \cap C^0(\overline \Omega; \mathbb{R}^m)$ solves the MSS in $\Omega$ in the viscosity sense, then $u$ is analytic in $\Omega$.
\end{conj}
\begin{remark}
     Since Lipschitz stationary solutions are viscosity solutions, there exist singular solutions beginning when $\Omega \subset \mathbb{R}^4$ (e.g. the Lawson-Osserman cone \cite{LO1}). 
\end{remark}

As evidence for a positive answer to Conjecture \ref{conj: Bernstein}, we prove a Bernstein-type theorem for homogeneous entire solutions on $\mathbb{R}^3$ without too much additional effort. Recall that full 2D regularity implies Lipschitz viscosity solutions are stationary when $n =2$ for each $m \geq 2$. Hence, the Bernstein Theorem for entire Lipschitz viscosity solutions holds for all $m \geq 2$ when $n = 2$ by \cite{CO67}. Since the proof of the Bernstein-type theorem is in some sense standard, we only sketch the details.
\begin{proposition}[Bernstein-type Theorem]\label{prop: bernstein}
    Suppose $u \in \Lip(\mathbb{R}^3; \mathbb{R}^m)$ is an entire viscosity solution to the MSS on $\mathbb{R}^3$. If $u$ is homogeneous, then $u$ is linear.
\end{proposition}
\begin{proof}
    Suppose $x_0 \in \mathbb{R}^3 \setminus \{0\}$ is a singular point for $u$. If $r_j \rightarrow 0^+$, consider the sequence of Lipschitz re-scalings of $u$
    \[u_{r_j}(y) := r_j^{-1}\big(u(x_0 + r_jy) - u(x_0)\big).\]
    By the Arzel\'a-Ascoli Theorem and Proposition \ref{prop: stability}, we can assume $u_{r_j} \rightarrow v$ locally uniformly for some entire Lipschitz viscosity solution $v$ on $\mathbb{R}^3$. Set $e_1 := x_0/|x_0|$ and complete this to a basis $\{e_1,e_2,e_3\}$ of $\mathbb{R}^3$. We can therefore define a new coordinate system $(t,y^\prime) \in \mathbb{R} \times e_1^\perp$, where $y^\prime$ is the coordinate on $e_1^\perp$. In these new coordinates, we have 
    \[v(t,y^\prime) = \frac{u(x_0)}{|x_0|} \cdot t + w(y^\prime)\]
    for some function $w \in \Lip(e_1^\perp; \mathbb{R}^m)$. Thus, $\mathscr{G}_v$ is invariant under translation by the line spanned by $(e_1, u(x_0)/|x_0|) \in \mathbb{R}^{3 + m}$. 
    
    Rotating the graph so that $\ell$ coincides with a horizontal coordinate axis, we obtain a $(3,0)$-set of the form $\mathbb{R} \times \mathscr{G}_{\tilde{w}}$ where $\tilde{w} \in \Lip(\mathbb{R}^2; \mathbb{R}^m)$. Abusing notation, we let $(t,y^\prime)$ also represent the coordinates on $\mathbb{R} \times \mathbb{R}^2$. If $\mathscr{G}_{\tilde{w}}$ is not a $(2,0)$-set, then there is a map $H \in C^2(\mathbb{R}^{2+m})$ and a point $p_0 \in \mathscr{G}_{\tilde{w}}$ such that $H \lvert_{\mathscr{G}_{\tilde{w}}}$ has a local maximum at $p_0$, $\nabla H(p_0) \neq 0$, and $\mu_1(p_0) + \mu_2(p_0) > 0$ where $\mu_1(p_0)$ and $\mu_2(p_0)$ are the two smallest eigenvalues of $D^2 H(p_0)$. Order the eigenvalues $\{\mu_i(p_0)\}_{i = 1}^{2+m}$ of $D^2 H$ at $p_0$ from least to greatest. Extending $H$ to a lift $h(t,y^\prime) := H(y^\prime)$ on $\mathbb{R}^3$, we see that $h \lvert_{\mathbb{R} \times \mathscr{G}_{\tilde{w}}}$ has a local maximum at $(0,p_0)$ with $\nabla h(0,p_0) \neq 0$. Furthermore, $D^2 h(0,p_0)$ has eigenvalues $\{0, \mu_i(p_0)\}_{i = 1}^{m+2}$. That is, the three smallest eigenvalues (not necessarily in order) of $D^2h$ at $(0,p_0)$ are $0,\, \mu_1(p_0)$, and $\mu_2(p_0)$. Since their sum is positive, this contradicts that $\mathbb{R} \times \mathscr{G}_{\tilde{w}}$ is a $(3,0)$-set. We conclude that $\mathscr{G}_{\tilde{w}}$ is a $(2,0)$-set so that $\tilde{w}$ is an entire Lipschitz viscosity solution to the MSS by Proposition \ref{prop: euivalent}.
    
    Now, $\tilde{w}$ is analytic by Theorem \ref{thm: 2d} so that the blow-up $v$ of $u$ is analytic as a map $\mathbb{R}^3 \rightarrow \mathbb{R}^m$. Theorem \ref{Allard} shows that $x_0$ is a regular point for $u$. Since $x_0 \neq 0$ was arbitrary, $\sing(u) \subset \{0\}$. It follows that $\sing(u)$ has finite $\mathcal{H}^0$ measure, implying $u$ is stationary by removal of singularities \cite[Corollary 4.9]{Sa18}. Since such a solution is analytic \cite{B79, FC1}, the proof is complete. 
\end{proof}

The primary obstruction to turning Proposition \ref{prop: bernstein} into a solution of Conjecture \ref{conj: Bernstein} is once again a lack of a monotonicity formula, since it is unclear whether blow-ups of a Lipschitz viscosity solution are homogeneous. If the answer to this question is positive, then Conjecture \ref{conj: Bernstein} holds by the standard blow-up scheme. In closing, we note that another important but under-developed line of inquiry is determining a proper existence framework.

\printbibliography

@article {A,
    AUTHOR = {Aronszajn, N.},
     TITLE = {A unique continuation theorem for solutions of elliptic
              partial differential equations or inequalities of second
              order},
   JOURNAL = {J. Math. Pures Appl. (9)},
  FJOURNAL = {Journal de Math\'ematiques Pures et Appliqu\'ees. Neuvi\`eme
              S\'erie},
    VOLUME = {36},
      YEAR = {1957},
     PAGES = {235--249},
      ISSN = {0021-7824,1776-3371},
   MRCLASS = {35.0X},
  MRNUMBER = {92067},
MRREVIEWER = {H.\ Bremekamp},
}

@article {San20,
    AUTHOR = {Santilli, Mario},
     TITLE = {Normal bundle and {A}lmgren's geometric inequality for
              singular varieties of bounded mean curvature},
   JOURNAL = {Bull. Math. Sci.},
  FJOURNAL = {Bulletin of Mathematical Sciences},
    VOLUME = {10},
      YEAR = {2020},
    NUMBER = {1},
     PAGES = {2050008, 24},
      ISSN = {1664-3607,1664-3615},
   MRCLASS = {49Q20 (35D40 49Q10 53A07 53C24)},
  MRNUMBER = {4095952},
MRREVIEWER = {Salvatore\ Stuvard},
       DOI = {10.1142/S1664360720500083},
       URL = {https://doi.org/10.1142/S1664360720500083},
}

@article {BD2,
    AUTHOR = {Dimler, Bryan},
     TITLE = {Partial regularity for {L}ipschitz solutions to the minimal
              surface system},
   JOURNAL = {Calc. Var. Partial Differential Equations},
  FJOURNAL = {Calculus of Variations and Partial Differential Equations},
    VOLUME = {62},
      YEAR = {2023},
    NUMBER = {9},
     PAGES = {Paper No. 260, 30},
      ISSN = {0944-2669,1432-0835},
   MRCLASS = {35J47 (35J50 35J60 49Q05 49Q15 49Q20)},
  MRNUMBER = {4660718},
MRREVIEWER = {Giorgio\ Saracco},
       DOI = {10.1007/s00526-023-02591-7},
       URL = {https://doi.org/10.1007/s00526-023-02591-7},
}

@article {Mo19,
    AUTHOR = {Mooney, Connor},
     TITLE = {A proof of the {K}rylov-{S}afonov theorem without
              localization},
   JOURNAL = {Comm. Partial Differential Equations},
  FJOURNAL = {Communications in Partial Differential Equations},
    VOLUME = {44},
      YEAR = {2019},
    NUMBER = {8},
     PAGES = {681--690},
      ISSN = {0360-5302,1532-4133},
   MRCLASS = {35J60 (35B65 35D40)},
  MRNUMBER = {3952774},
MRREVIEWER = {Armin\ Schikorra},
       DOI = {10.1080/03605302.2019.1581807},
       URL = {https://doi.org/10.1080/03605302.2019.1581807},
}

@article {ASS12,
    AUTHOR = {Armstrong, Scott N. and Silvestre, Luis E. and Smart, Charles
              K.},
     TITLE = {Partial regularity of solutions of fully nonlinear, uniformly
              elliptic equations},
   JOURNAL = {Comm. Pure Appl. Math.},
  FJOURNAL = {Communications on Pure and Applied Mathematics},
    VOLUME = {65},
      YEAR = {2012},
    NUMBER = {8},
     PAGES = {1169--1184},
      ISSN = {0010-3640,1097-0312},
   MRCLASS = {35J60 (35B65 35D40)},
  MRNUMBER = {2928094},
MRREVIEWER = {Jos\'e\ Carmona Tapia},
       DOI = {10.1002/cpa.21394},
       URL = {https://doi.org/10.1002/cpa.21394},
}

@book {CC95,
    AUTHOR = {Caffarelli, Luis A. and Cabr\'e, Xavier},
     TITLE = {Fully nonlinear elliptic equations},
    SERIES = {American Mathematical Society Colloquium Publications},
    VOLUME = {43},
 PUBLISHER = {American Mathematical Society, Providence, RI},
      YEAR = {1995},
     PAGES = {vi+104},
      ISBN = {0-8218-0437-5},
   MRCLASS = {35J60 (35-01 35B45 35B65 35Dxx)},
  MRNUMBER = {1351007},
MRREVIEWER = {P.\ Lindqvist},
       DOI = {10.1090/coll/043},
       URL = {https://doi.org/10.1090/coll/043},
}

@misc{HMT23,
      title={On the Lawson-Osserman conjecture}, 
      author={Jonas Hirsch and Connor Mooney and Riccardo Tione},
      year={2023},
      eprint={2308.04997},
      archivePrefix={arXiv},
      primaryClass={math.AP},
      url={https://arxiv.org/abs/2308.04997},
      note={Invent. Math., to appear. arXiv preprint: 2308.04997}
}

@article {Du86,
    AUTHOR = {Duggan, J. P.},
     TITLE = {Regularity theorems for varifolds with mean curvature},
   JOURNAL = {Indiana Univ. Math. J.},
  FJOURNAL = {Indiana University Mathematics Journal},
    VOLUME = {35},
      YEAR = {1986},
    NUMBER = {1},
     PAGES = {117--144},
      ISSN = {0022-2518,1943-5258},
   MRCLASS = {49F20},
  MRNUMBER = {825631},
MRREVIEWER = {Vanna\ Zanelli},
       DOI = {10.1512/iumj.1986.35.35006},
       URL = {https://doi.org/10.1512/iumj.1986.35.35006},
}

@article {CO67,
    AUTHOR = {Chern, Shiing-Shen and Osserman, Robert},
     TITLE = {Complete minimal surfaces in Euclidean {$n$}-space},
   JOURNAL = {J. Analyse Math.},
  FJOURNAL = {Journal d'Analyse Math\'ematique},
    VOLUME = {19},
      YEAR = {1967},
     PAGES = {15--34},
      ISSN = {0021-7670,1565-8538},
   MRCLASS = {53.04},
  MRNUMBER = {226514},
MRREVIEWER = {H.\ B.\ Jenkins},
       DOI = {10.1007/BF02788707},
       URL = {https://doi.org/10.1007/BF02788707},
}

@article {FC1,
    AUTHOR = {Fischer-Colbrie, D.},
     TITLE = {Some rigidity theorems for minimal submanifolds of the sphere},
   JOURNAL = {Acta Math.},
  FJOURNAL = {Acta Mathematica},
    VOLUME = {145},
      YEAR = {1980},
    NUMBER = {1-2},
     PAGES = {29--46},
      ISSN = {0001-5962,1871-2509},
   MRCLASS = {53C42 (49F10 58E12)},
  MRNUMBER = {558091},
MRREVIEWER = {Jonathan\ Sacks},
       DOI = {10.1007/BF02414184},
       URL = {https://doi.org/10.1007/BF02414184},
}

@incollection {Sa18,
    AUTHOR = {Savin, O.},
     TITLE = {Viscosity solutions and the minimal surface system},
 BOOKTITLE = {Nonlinear analysis in geometry and applied mathematics. {P}art
              2},
    SERIES = {Harv. Univ. Cent. Math. Sci. Appl. Ser. Math.},
    VOLUME = {2},
     PAGES = {135--145},
 PUBLISHER = {Int. Press, Somerville, MA},
      YEAR = {2018},
      ISBN = {978-1-57146-359-3},
   MRCLASS = {35J93 (35B65 35D40 53C42)},
  MRNUMBER = {3823884},
MRREVIEWER = {Naoki\ Yamada},
}

@article {Sm93,
    AUTHOR = {Smale, Nathan},
     TITLE = {Geometric {P}.{D}.{E}.s with isolated singularities},
   JOURNAL = {J. Reine Angew. Math.},
  FJOURNAL = {Journal f\"ur die Reine und Angewandte Mathematik. [Crelle's
              Journal]},
    VOLUME = {440},
      YEAR = {1993},
     PAGES = {1--41},
      ISSN = {0075-4102,1435-5345},
   MRCLASS = {58G20 (53A10 58E20 58G18)},
  MRNUMBER = {1225956},
MRREVIEWER = {Libin\ Mou},
       DOI = {10.1515/crll.1993.440.1},
       URL = {https://doi.org/10.1515/crll.1993.440.1},
}

@book {N,
    AUTHOR = {Nicolaescu, Liviu I.},
     TITLE = {Lectures on the geometry of manifolds},
   EDITION = {Third},
 PUBLISHER = {World Scientific Publishing Co. Pte. Ltd., Hackensack, NJ},
      YEAR = {[2021] \copyright 2021},
     PAGES = {xviii+682},
      ISBN = {978-981-121-482-0; 978-981-121-481-3; 978-981-121-595-7},
   MRCLASS = {53-02 (58-02)},
  MRNUMBER = {4228252},
}

@article {SS14,
    AUTHOR = {Silvestre, Luis and Sirakov, Boyan},
     TITLE = {Boundary regularity for viscosity solutions of fully nonlinear
              elliptic equations},
   JOURNAL = {Comm. Partial Differential Equations},
  FJOURNAL = {Communications in Partial Differential Equations},
    VOLUME = {39},
      YEAR = {2014},
    NUMBER = {9},
     PAGES = {1694--1717},
      ISSN = {0360-5302,1532-4133},
   MRCLASS = {35J60 (35B65 35D40 35J25 35J67)},
  MRNUMBER = {3246039},
MRREVIEWER = {Leonard\ Monsaingeon},
       DOI = {10.1080/03605302.2013.842249},
       URL = {https://doi.org/10.1080/03605302.2013.842249},
}

@article {Je96,
    AUTHOR = {Jensen, Robert R.},
     TITLE = {Uniformly elliptic {PDE}s with bounded, measurable
              coefficients},
   JOURNAL = {J. Fourier Anal. Appl.},
  FJOURNAL = {The Journal of Fourier Analysis and Applications},
    VOLUME = {2},
      YEAR = {1996},
    NUMBER = {3},
     PAGES = {237--259},
      ISSN = {1069-5869,1531-5851},
   MRCLASS = {35J15 (35D10)},
  MRNUMBER = {1379505},
MRREVIEWER = {Ya.\ A.\ Ro\u itberg},
       DOI = {10.1007/s00041-001-4031-6},
       URL = {https://doi.org/10.1007/s00041-001-4031-6},
}

@book {GM12,
    AUTHOR = {Giaquinta, Mariano and Martinazzi, Luca},
     TITLE = {An introduction to the regularity theory for elliptic systems,
              harmonic maps and minimal graphs},
    SERIES = {Appunti. Scuola Normale Superiore di Pisa (Nuova Serie)
              [Lecture Notes. Scuola Normale Superiore di Pisa (New
              Series)]},
    VOLUME = {11},
   EDITION = {Second},
 PUBLISHER = {Edizioni della Normale, Pisa},
      YEAR = {2012},
     PAGES = {xiv+366},
      ISBN = {978-88-7642-442-7; 978-88-7642-443-4},
   MRCLASS = {35-02 (35B65 35J20 35J60 58E20)},
  MRNUMBER = {3099262},
       DOI = {10.1007/978-88-7642-443-4},
       URL = {https://doi.org/10.1007/978-88-7642-443-4},
}

@book {GT01,
    AUTHOR = {Gilbarg, David and Trudinger, Neil S.},
     TITLE = {Elliptic partial differential equations of second order},
    SERIES = {Classics in Mathematics},
      NOTE = {Reprint of the 1998 edition},
 PUBLISHER = {Springer-Verlag, Berlin},
      YEAR = {2001},
     PAGES = {xiv+517},
      ISBN = {3-540-41160-7},
   MRCLASS = {35-02 (35Jxx)},
  MRNUMBER = {1814364},
}

@misc{Dim25,
      title={Minimal submanifolds with multiple isolated singularities}, 
      author={Bryan Dimler},
      year={2025},
      eprint={2409.20327},
      archivePrefix={arXiv},
      primaryClass={math.DG},
      url={https://arxiv.org/abs/2409.20327}, 
     note={arXiv preprint: 2409.20327}
}

@article {LO1,
    AUTHOR = {Lawson, Jr., H. B. and Osserman, R.},
     TITLE = {Non-existence, non-uniqueness and irregularity of solutions to
              the minimal surface system},
   JOURNAL = {Acta Math.},
  FJOURNAL = {Acta Mathematica},
    VOLUME = {139},
      YEAR = {1977},
    NUMBER = {1-2},
     PAGES = {1--17},
      ISSN = {0001-5962,1871-2509},
   MRCLASS = {35J60 (53A10 58E12)},
  MRNUMBER = {452745},
       DOI = {10.1007/BF02392232},
       URL = {https://doi.org/10.1007/BF02392232},
}

@article {Wh16,
    AUTHOR = {White, Brian},
     TITLE = {Controlling area blow-up in minimal or bounded mean curvature
              varieties},
   JOURNAL = {J. Differential Geom.},
  FJOURNAL = {Journal of Differential Geometry},
    VOLUME = {102},
      YEAR = {2016},
    NUMBER = {3},
     PAGES = {501--535},
      ISSN = {0022-040X,1945-743X},
   MRCLASS = {53C42},
  MRNUMBER = {3466806},
MRREVIEWER = {Yijun\ He},
       URL = {http://projecteuclid.org/euclid.jdg/1456754017},
}

@incollection {Si84,
    AUTHOR = {Simon, Leon},
     TITLE = {Isolated singularities of extrema of geometric variational
              problems},
 BOOKTITLE = {Harmonic mappings and minimal immersions ({M}ontecatini,
              1984)},
    SERIES = {Lecture Notes in Math.},
    VOLUME = {1161},
     PAGES = {206--277},
 PUBLISHER = {Springer, Berlin},
      YEAR = {1985},
      ISBN = {3-540-16040-X},
   MRCLASS = {58E15 (58E20)},
  MRNUMBER = {821971},
MRREVIEWER = {Harold\ Parks},
       DOI = {10.1007/BFb0075139},
       URL = {https://doi.org/10.1007/BFb0075139},
}

@article {B79,
    AUTHOR = {Barbosa, Jo{\~a}o Lucas Marqu\^es},
     TITLE = {An extrinsic rigidity theorem for minimal immersions from
              {$S\sp{2}$}\ into {$S\sp{n}$}},
   JOURNAL = {J. Differential Geometry},
  FJOURNAL = {Journal of Differential Geometry},
    VOLUME = {14},
      YEAR = {1979},
    NUMBER = {3},
     PAGES = {355--368},
      ISSN = {0022-040X,1945-743X},
   MRCLASS = {53C42 (53C45)},
  MRNUMBER = {594706},
MRREVIEWER = {V.\ T.\ Fomenko},
       URL = {http://projecteuclid.org/euclid.jdg/1214435101},
}

@article {NT23,
    AUTHOR = {Nascimento, Thialita M. and Teixeira, Eduardo V.},
     TITLE = {New regularity estimates for fully nonlinear elliptic
              equations},
   JOURNAL = {J. Math. Pures Appl. (9)},
  FJOURNAL = {Journal de Math\'ematiques Pures et Appliqu\'ees. Neuvi\`eme
              S\'erie},
    VOLUME = {171},
      YEAR = {2023},
     PAGES = {1--25},
      ISSN = {0021-7824,1776-3371},
   MRCLASS = {35B65 (35D40 35J60)},
  MRNUMBER = {4545167},
MRREVIEWER = {Liangwei\ Wang},
       DOI = {10.1016/j.matpur.2022.12.008},
       URL = {https://doi.org/10.1016/j.matpur.2022.12.008},
}

@misc{DG26,
      title={Special Lagrangians with multiple isolated singularities}, 
      author={Bryan Dimler and Filippo Gaia},
      year={2026},
      eprint={2607.26302},
      archivePrefix={arXiv},
      primaryClass={math.DG},
      url={https://arxiv.org/abs/2607.26302},
      note= {arXiv preprint: 2607.26302},
}
\end{document}